\documentclass[12pt]{amsart}
\usepackage{latexsym,amsmath,amsfonts,amscd,amssymb,mathtools}
\usepackage{graphics}
\usepackage{xcolor}
\newtheorem{theorem}{Theorem}[section]
\newtheorem{lemma}[theorem]{Lemma}
\newtheorem{proposition}[theorem]{Proposition}
\newtheorem{corollary}[theorem]{Corollary}

\newtheorem{definition}[theorem]{Definition}

\newtheorem{example}[theorem]{Example}

\theoremstyle{remark}

\newcommand{\bld}{\boldsymbol}

\newcommand{\Res}{\text{Res}}

\newcommand{\Div}{\operatorname{Div}}

\newcommand{\cC}{{\mathcal C}}

\newcommand{\cE}{{\mathcal E}}

\newcommand{\CC}{{\mathbb C}}

\newcommand{\NN}{{\mathbb N}}

\newcommand{\QQ}{{\mathbb Q}}
\newcommand{\RR}{{\mathbb R}}

\newcommand{\ZZ}{{\mathbb Z}}

\renewcommand{\Re}{\operatorname{Re}}
\renewcommand{\Im}{\operatorname{Im}}

\renewcommand{\a}{\alpha}
\renewcommand{\b}{\beta}

\newcommand{\eps}{\epsilon}

\renewcommand{\o}{\omega}

\title{On the definition of higher Gamma functions}


\author[R. P\'{e}rez-Marco]{Ricardo P\'{e}rez-Marco}
\address{CNRS, IMJ-PRG, Universit\'e Paris Cit\'e, Paris, France}
\email{ricardo.perez.marco@gmail.com}

\begin{document}

\begin{abstract}
We generalize our previous new definition of Euler Gamma function to higher Gamma functions. With this unified approach, 
we characterize Barnes higher Gamma functions, Mellin Gamma functions, Barnes multiple Gamma functions, Jackson $q$-Gamma function, 
and Nishizawa higher $q$-Gamma functions. The method presented extends to more general 
functional equations. This generalization reveals 
the multiplicative group structure of solutions of the functional equation that appears as a cocycle equation.
We also generalize Barnes hierarchy of higher Gamma function and multiple Gamma functions. 
With this new definition, Barnes-Hurwitz zeta functions 
are no longer needed in the definition of Barnes multiple Gamma functions. 
This simplifies the classical definition, without the 
necessary analytic preliminaries about the meromorphic extension of Barnes-Hurwitz zeta functions,
and allows to define a larger class of Gamma functions. For some algebraic independence conditions 
on the parameters, we do have uniqueness 
of the solutions. This implies the identification of a subclass of our multiple 
Gamma functions with classical Barnes multiple Gamma functions.
\end{abstract}

\maketitle

%%%%%%%%%%%%%%%%%%%%%%%%%%%%%%%%
\section{Introduction}
%%%%%%%%%%%%%%%%%%%%%%%%%%%%%%%%

The first result is a new characterization and  definition of Euler Gamma function that was 
first presented in the article \cite{PM1}
dedicated to Euler Gamma function. 
We can develop in a natural way the classical formulas in the theory from this new definition. 
We denote the right half 
complex plane by $\CC_+=\{s\in \CC ; \Re s >0\}$.

\begin{theorem}\label{thm:main-Euler-gamma}
There is one and only one finite order meromorphic function $\Gamma(s)$, $s\in \CC$, 
without zeros nor poles in $\CC_+$, with $\Gamma(1)=1$, 
$\Gamma'(1)\in \RR$, that satisfies the functional equation
 $$
 \Gamma(s+1)=s\, \Gamma(s)
 $$
\end{theorem}

\medskip

\begin{definition}[Euler Gamma function]
The only solution to the above conditions is the Euler Gamma function.
\end{definition}

Without the condition $\Gamma'(1)\in \RR$ we don't have uniqueness, but we have the following result:

\begin{theorem}\label{thm:main-Euler-gamma_bis}
Let $f$ be a finite order meromorphic function in $\CC$, 
without zeros nor poles in $\CC_+$,  and satisfies the functional equation
 $$
 f(s+1)=s\, f(s)  \ ,
 $$
 then there exists $a\in \ZZ$ and $b\in \CC$ such that 
 $$
 f(s)=e^{2\pi i a s+b} \Gamma (s) \ .
 $$
 Moreover, if $f(1)=1$ then we have
 $$
 f(s)=e^{2\pi i a s} \Gamma (s) \ .
 $$
 
\end{theorem}

The proof can be found in \cite{PM1}, but we reproduce it here as a preliminary result for the generalizations 
which are the core of this article. We refer to the companion article \cite{PM1} for 
the various definitions of Euler Gamma function and the historical development of the subject of Eulerian integrals. 
We strongly encourage the reader to study first \cite{PM1}, and also consult 
the bibliographic notes in \cite{PM2} before studying this article.

In the proof we use the elementary theory of entire function that can be found in classical textbooks as \cite{Boa} 
(this is recalled in the Appendix of \cite{PM1}). 
 
\begin{proof}

We first prove existence and then uniqueness.

\smallskip

\textbf{Existence:}
If we have a  function satisfying the previous conditions then its divisor must be contained in $\CC-\CC_+$,
and the functional equation implies that it has no zeros and only simple poles at the non-positive integers. 
We can directly construct such a meromorphic function $g$
with such divisor, for example,
\begin{equation}\label{eq:W-solution}
g(s)=s^{-1}\prod_{n=1}^{+\infty} \left ( 1+\frac{s}{n}\right )^{-1}e^{s/n} 
\end{equation}
which converges since $\sum_{n\geq 1} n^{-2} <+\infty$, and is a meromorphic function of finite order.
Now, we have that the meromorphic function $\frac{g(s+1)}{s g(s)}$ has no zeros nor poles and it is also of finite 
order (as ratio of finite order meromorphic functions), hence there exists a polynomial $P$ such that 
$$
\frac{g(s+1)}{s g(s)}=e^{P(s)} \ .
$$
Consider a polynomial $Q$ such that 
\begin{equation}\label{eq:W-solution2}
\Delta Q(s) =Q(s+1)-Q(s) =P(s)
\end{equation}
The polynomial $Q$ is uniquely determined from $P$ up to a constant, hence we can freely choose $Q(1)$.
By the product definition, $g$ is real and positive for positive real numbers, and we choose
$Q(1)=\log g(1)$. Now we have that  $\Gamma (s) = e^{-Q(s)}g(s)$ satisfies the 
functional equation, $\Gamma(1)=1$, and $\Gamma'(1) \in \RR$ because $\Gamma$ is real analytic in the 
positive reals. Thus we proved existence.

\smallskip

\textbf{Uniqueness:}
Consider a second solution $f$. Let $F(s) = \Gamma(s)/f(s)$. Then $F$ is an entire function of finite order 
without zeros, hence we can write $F(s)=\exp{A(s)}$
for some polynomial $A$. Moreover, the functional equation shows that $F$ is $\ZZ$-periodic. Hence, there exists 
an integer $a\in \ZZ$,  such that for any $s\in \CC$,
$$
A(s+1)=A(s)+2\pi i a \ .
$$
It follows that $A(s)=2\pi i a s+b$ for some $b\in \CC$. Since $F(1)=1$, we have $e^b=1$. Since $F'(1) \in \RR$, and 
$F'(1)=F'(1)/F(1)=2 \pi i a  \in \RR$ we have $a=0$, thus $F$ is constant, $F\equiv 1$ and $f=\Gamma$.
\end{proof}

\textbf{Remarks.}

\begin{itemize}
 \item Using the functional equation we can weaken the conditions 
 and request only that the function is meromorphic only on $\CC_+$ with the corresponding 
 finite order growth. We can also 
 assume that it is only defined on a cone containing the positive real axes, a  
 vertical strip of width larger than $1$, or in general with any region $\Omega$ which is a 
transitive region for the integer translations and $f$ satisfies the finite order growth condition in 
$\Omega$ when $s\to +\infty$. 
\begin{proposition}\label{prop:domain}
Let $\Omega\subset \CC$ a domain such that for any $s\in \CC$ there exists an integer $n(s)\in \ZZ$ 
such that $s+n(s)\in \Omega$, and $|n(s)|\leq C |s|^d$, for some constants $C, d>0$ depending 
only on $\Omega$. Then any function $\tilde \Gamma$ satisfying a finite order estimate in $\Omega$ and the functional 
equation $\tilde \Gamma(s+1)=s\tilde \Gamma(s)$ when $s, s+1\in \Omega$, 
extends to a finite order meromorphic function on $\CC$.
\end{proposition}

\begin{proof}
Let $\tilde \Gamma$ be such a function. Let $\Omega$ be corresponding region.
Iterating the functional equation we get that $\tilde \Gamma$ extends meromorphically to the whole complex plane. 
Then, if $g$ is the Weierstrass product (\ref{eq:W-solution}) and $Q$ a polynomial given by (\ref{eq:W-solution2}), 
the function $h(s)=\tilde \Gamma(s)/(e^{-Q(s)}g(s))$ 
is a $\ZZ$-periodic entire function. Since $1/(e^{-Q}g)$ is an entire function of finite order, we have in $\Omega$ 
the finite order estimate for $h$. Using that $|n(s)|\leq C|s|^d$, we get that $h$ is of 
finite order, hence $\tilde \Gamma$ is meromorphic and of finite order
in the plane.
\end{proof}

\item Assuming $\Gamma$ real-analytic we get $\Gamma'(1)\in \RR$, but this last  
condition is much weaker. Also, as it follows from the proof, we can 
replace this condition by $\Gamma(a)\in  \RR$ for some $a\in \RR-\ZZ$, or only request that $\Gamma$ is asymptotically real, 
$\lim_{x\in \RR, x\to +\infty} \Im \Gamma (x)=0$. 
Without the condition $\Gamma'(1)\in \RR$
the proof shows that $\Gamma$ is uniquely determined up to a factor $e^{2\pi i k s}$ with $k\in \ZZ$. 
\end{itemize}

\section{General definition.}

\subsection{General definition and characterization.}

We first need to recall the notion of \textit{``Left Located Divisor''} (LLD) function that is useful in the theory 
of Poisson-Newton formula for finite order meromorphic functions (\cite{MPM1}, \cite{MPM2}).

\begin{definition}[LLD function]
A meromorphic function $f$ in $\CC$ is in the class LLD (Left Located Divisor) if  $f$ has no zeros nor poles 
in $\CC_+$, \textit{i.e.} $\Div(f) \subset \CC-\CC_+$. 

The function is in the class CLD (Cone Located Divisor) 
if its divisor is contained in a closed cone in the closed half plane $\CC-\CC_+$.
\end{definition}

The following Theorem is a generalization of Theorem \ref{thm:main-Euler-gamma} which results for the simple LLD function $f(s)=s$.

\begin{theorem}\label{thm:general1}
Let $f$ be a real analytic LLD meromorphic function in $\CC$ of finite order. 
There exists a unique function $\Gamma^f$, the Gamma function  associated to $f$, satisfying the following properties:
\begin{enumerate}
%   \item $\Gamma^f_0(s)=f(s)$,\label{GC1}
  \item $\Gamma^f(1)=1$ , \label{GC1}
  \item $\Gamma^f(s+1)=f(s)\Gamma^f(s)$ ,\label{GC2}
  \item $\Gamma^f$ is a meromorphic function of finite order, \label{GC3}
  \item $\Gamma^f$ is LLD,\label{GC4}
  \item $\Gamma^f$ is real analytic.\label{GC5}
\end{enumerate} 
If $f$ is CLD then $\Gamma^f$ is CLD.
\end{theorem}

\begin{proof}

The proof follows the same lines as the proof of Theorem \ref{thm:main-Euler-gamma}. First, we prove that the functional 
equation (\ref{GC2}) determines the divisor of $\Gamma^f$, then we construct a solution using a Weierstrass 
product, and finally we prove uniqueness.

\medskip
$\bullet$ \textbf{Determination of the divisor.}
\medskip

As usual, we denote the divisor of $f$ as the formal sum
$$
\Div (f)=\sum_\rho n_\rho (f) . (\rho)
$$
where the sum is extended over $\rho\in \CC$ and $n_\rho (f)$ is the multiplicity of the zero if $\rho$ is a zero, the negative multiplicity of the 
pole if $\rho$ is a pole, or $n_\rho (f)=0$ if $\rho$ is neither a zero or pole. 
Hence the function $n_\rho$ is an almost zero function. 
A divisor is said to be LLD, resp. CLD, if it is the 
divisor of a LLD, resp. CLD, function.

\begin{lemma}\label{lemma:divisor_determination}
If $\Gamma^f$ is LLD and satisfies the functional equation (\ref{GC2}), then the divisor of $\Gamma^f$ is
$$
\Div (\Gamma^f) =-\sum_{\rho, k\geq 0} n_\rho(f)  \cdot (\rho -k)
$$
where 
$$
\Div (f) = \sum_{\rho} n_\rho(f) \cdot (\rho)
$$
and 
\begin{equation}\label{eq:multiplicities}
n_\rho (\Gamma^f) = -\sum_{k=0}^{|\rho | } n_{\rho +k}(f)
\end{equation}

If the divisor $\Div (f)$ is LLD, resp. CLD, then $\Div (\Gamma^f)$ is LLD, resp. CLD.
\end{lemma}

We allow ourselves the slight abuse of notation $\Div(\Gamma^f)$ to denote the divisor of a 
potential solution $\Gamma^f$ when we have not yet proved the existence of $\Gamma^f$.

\begin{proof}
For any $\rho \in \CC$, the functional equation gives
$$
n_{\rho +1} (\Gamma^f) = n_\rho(f) +n_\rho (\Gamma^f) \ ,
$$
or equivalently
$$
n_{\rho } (\Gamma^f) = -n_\rho(f) +n_{\rho +1} (\Gamma^f) \ ,
$$
Hence, by induction, we have for $m\geq 1$,
$$
n_\rho (\Gamma^f) = -\sum_{k=0}^{m-1} n_{\rho +k}(f) +n_{\rho+m}(\Gamma^f) 
$$
and since $\Gamma^f$ is LLD, for $m\geq |\rho| \geq |\Re \rho| =-\Re \rho$ 
we have  $n_{\rho+m}(\Gamma^f) =0$, so
\begin{equation}
n_\rho (\Gamma^f) = -\sum_{k=0}^{|\rho |-1} n_{\rho +k}(f) =-\sum_{k=0}^{+\infty } n_{\rho +k}(f)    
\end{equation}
and we get, with $\rho'=\rho+k$,
$$
\Div(\Gamma^f) =-\sum_{k\geq 0} \sum_{\rho } n_{\rho +k} \cdot (\rho) = -\sum_{\rho' , k\geq 0} n_{\rho'} \cdot (\rho'-k) 
$$
which gives the formula for $\Div (\Gamma^f)$.
\end{proof}

\newpage

$\bullet$ \textbf{Convergence exponent of a divisor.}
\medskip

\begin{definition}
The divisor of $f$  has exponent of convergence $\alpha >0$ if
$$
||\Div(f)||_\alpha=\sum_{\rho\not=0} |n_\rho(f)| . |\rho |^{-\alpha} <+\infty  \ .
$$
\end{definition}

We recall that a meromorphic function of finite order has a divisor with some  finite 
exponent of convergence. More precisely, 
if $o(f)<+\infty $ is the order of $f$, then for any $\eps>0$, $\alpha=o(f)+\eps$ 
is an exponent of convergence of its divisor.

\begin{proposition}\label{prop:conv_exponent1}
If  $\Div(f)$ is LLD of finite order, then $\Div(\Gamma^f)$ given by 
Lemma \ref{lemma:divisor_determination} is LLD and of finite order.

More precisely, if $\Div(f)$ is of oder $\alpha >0$ then $\Div(\Gamma^f)$ is of order $\alpha +1$.
\end{proposition}

We prove first a more precise result when $f$ is in the class CLD.

\begin{proposition}\label{prop:conv_exponent2}
If $\alpha >0$ is an exponent of convergence for $f$ in the class CLD, then $\Gamma^f$ is CLD and
$\alpha +1$ is an exponent of convergence for $\Gamma^f$. More precisely, there exists a constant $C>0$ 
such that 
$$
||\Div(\Gamma^f)||_{\alpha+1} \leq  C \, ||\Div(f)||_{\alpha+1} + \frac{C}{\alpha} \, ||\Div(f)||_{\alpha}
$$

\end{proposition}

\begin{proof}
Lemma \ref{lemma:divisor_determination} proves that $\Gamma^f$ is CLD if we start with $f$ CLD.
Now, if $f$ is CLD, there is a constant $C>0$ such that for any $k\geq 1$ and 
$\rho$ in the left cone (the constant $C$ depends on the cone)
$$
|\rho -k|^{-1}\leq C (|\rho| +k)^{-1} \ .
$$
Then we have, with $\rho'=\rho+k$,
\begin{align*}
||\Div(\Gamma^f)||_\beta &=\sum_{\rho \not= 0} |n_\rho (\Gamma^f)| . |\rho|^{-\beta} \\ 
&=  \sum_{\rho \not= 0} \sum_{k=0}^{|\rho|} |n_{\rho +k}(f)|. |\rho|^{-\beta} \\
&= \sum_{\rho' \notin \NN} |n_{\rho'}(f)| \sum_{k=0}^{+\infty} |\rho'-k|^{-\beta} \\
&\leq C  \sum_{\rho' \not= 0} |n_{\rho'}(f)| \sum_{k=0}^{+\infty} (|\rho'| +k)^{-\beta} \\
&= C  \sum_{\rho' \not= 0} |n_{\rho'}(f)| . |\rho'|^{-\beta} + C \sum_{\rho' \not= 0} |n_{\rho'}(f)| \int_0^{+\infty} (|\rho'| +x)^{-\beta} \, dx \\
&\leq  C  \sum_{\rho' \not= 0} |n_{\rho'}(f)| . |\rho'|^{-\beta} + \frac{C}{\beta -1} \sum_{\rho' \not= 0} |n_{\rho'}(f)| . |\rho'|^{-\beta+1}\\
& = C \, ||\Div(f)||_\beta + \frac{C}{\beta -1} \, ||\Div(f)||_{\beta-1}
\end{align*}
hence, for $\beta =\alpha +1$ the sum is converging and we prove the Lemma.
\end{proof}

Now we prove Proposition \ref{prop:conv_exponent1}.

\begin{proof}
We consider the part of the divisor of $\Gamma^f$ contained in the cone $\cC_0=\{ |\Im \rho | < -\Re \rho \}$. The same argument as before proves that
\begin{equation}\label{eq:intermediate}
\sum_{\rho \in \cC_0} |n_\rho (\Gamma^f)|. |\rho|^{-(\alpha +1)} < +\infty 
\end{equation}
This can be seen by observing that $\Div(\Gamma^f)$ is generated by families $\rho, \rho-1, \rho-2,\ldots$ where $\rho \in \Div(f)$. The part of this sequence
contained in the cone $\cC_0$ is of the form $\rho_0, \rho_0-1, \rho_0-2,\ldots$, where $\rho_0=\rho-k_0$ for some integer $k_0=k_0(\rho)\geq 0$. Now, since $\Re \rho <0$, we have
$|\rho_0|^{-\alpha} =|\rho-k_0|^{-\alpha} \leq |\rho|^{-\alpha}$. Hence, if we consider the meromorphic function $f_0$ with divisor generated by the $(\rho_0)$, we do have
\begin{equation*}
 ||\Div(f_0)||_{\alpha} \leq ||\Div(f)||_{\alpha}
\end{equation*}
and we use Proposition \ref{prop:conv_exponent2} to prove (\ref{eq:intermediate}).

Now, we have to care about the convergence of the remaining sum and prove
\begin{equation}
\sum_{\rho \notin \cC_0} |n_\rho (\Gamma^f)|. |\rho|^{-(\alpha +1)} < +\infty 
\end{equation}
Any such divisor point of $\Gamma^f$ is generated by a zero or pole $\rho \notin \cC_0$ of $f$, 
that generates $\rho, \rho-1, \rho-2,\ldots$. There are at most $|\rho|$ such points associated to each $\rho$. Hence we can bound the sum
\begin{equation*}
 \sum_{\rho \notin \cC_0} |n_\rho (\Gamma^f)| |\rho|^{-(\alpha +1)} \leq \sum_\rho |n_\rho (f)| |\rho|.|\rho|^{-\alpha } \leq ||\Div(f)||_{\alpha+1}
\end{equation*}
and the result follows.

\end{proof}

\medskip
% \newpage
$\bullet$ \textbf{Existence of $\Gamma^f$.}
\medskip

Since $f$ has finite order, the divisor of $f$ has a finite convergence exponent. Hence, 
$\Div(\Gamma^f)$ determined by Lemma \ref{lemma:divisor_determination} has a finite exponent of convergence. 
Let $d\geq 1$ be an integer that is an exponent of convergence for this divisor (the case $d=0$ only occurs for a finite divisor). 
We consider the Weierstrass product,
$$
g(s)=s^{-n_0(f)} \prod_{\rho \not= 0} E_d(s/\rho)^{n_\rho (\Gamma^f)}
$$
where 
$$
E_d(x)= (1-x)\exp \left ( x+\frac{x^2}{2}+\ldots +\frac{x^d}{d} \right ) \ .
$$
Then $g$ has order $d$ and $\Div (g) =\Div(\Gamma^f)$. Therefore the meromorphic function
$$
\frac{g(s+1)}{f(s) g(s)}
$$
is of finite order and has no zeros nor poles. So, it is an entire function of finite order without zeros. 
Therefore, there exists a polynomial $\phi$ such that 

\begin{equation}\label{eqn:Gammaf2v2}
 \frac{g(s+1)}{f(s) g(s)}=e^{\phi(s)} 
 \end{equation}
There is a unique polynomial $\psi$ such that $\psi(0)=0$ and
\begin{equation}\label{eqn:Gammaf5v2}
 \psi(s+1)-\psi(s)=\phi(s).
\end{equation}
We can obtain $\psi$ directly by developing $\phi$ on the bases of falling factorial polynomials,
$s^{\underline k} =s(s-1)\ldots (s-k+1)$, that diagonalize the difference operator, $\Delta s^{\underline k} = k \, s^{\underline {k-1}}$, 
$$
\phi(s)=\sum_{k=0}^{n} \frac{a_k}{k!}\, s^{\underline k}
$$
then 
$$
\psi (s)=\sum_{k=0}^{+\infty} \frac{a_k}{(k+1)!}\, s^{\underline {k+1}} \ .
$$
Now, considering a constant $c$ such that $e^c=e^{-\phi(0)}g(1)^{-1}$ the meromorphic function
\begin{equation}\label{eqn:Gammaf3v2}
 \Gamma^f(s)=e^{\psi(s)+c} g(s),
 \end{equation}
satisfies $\Gamma^f(1)=1$ (condition (\ref{GC1})), the functional equation (\ref{GC2}) and 
all the other conditions in Theorem \ref{thm:general1}, and we have proved the existence.

\medskip
% \newpage

$\bullet$ \textbf{Uniqueness of $\Gamma^f$.}
\medskip

 Consider a second solution $G$. Let $F(s) = \Gamma^f(s)/G(s)$. Then $F$ is an entire function of finite order 
 without zeros, hence we can write $F(s)=\exp{A(s)}$
 for some polynomial $A$. Moreover, the functional equation shows that $F$ is $\ZZ$-periodic. Therefore, there exists 
 an integer $a\in \ZZ$,  such that for any $s\in \CC$,
 $$
 A(s+1)=A(s)+2\pi i a \ .
 $$
 It follows that $A(s)=2\pi i a s+b$ for some $b\in \CC$. Since $F(1)=1$, we have $e^b=1$. Since $F'(1) \in \RR$, and 
 $F'(1)=F'(1)/F(1)=2 \pi i a  \in \RR$ we have $a=0$, thus $F$ is constant, $F\equiv 1$ and $G=\Gamma^f$.

\end{proof}

\medskip

\subsection{Uniqueness results.}

It is interesting to note, following the argument for uniqueness, that we can drop the 
normalisation condition (\ref{GC1}) and the real-analyticity condition (\ref{GC5}) and we obtain the following Theorem
(this is similar to Theorem \ref{thm:main-Euler-gamma_bis}),

\begin{theorem}\label{thm:general1_bis}
Let $f$ be a LLD meromorphic function in $\CC$ of finite order. 
We consider a function $g$ satisfying
\begin{enumerate}
  \item $g(s+1)=f(s)g(s)$ ,\label{GC2_bis}
  \item $g$ is a meromorphic function of finite order, \label{GC3_bis}
  \item $g$ is LLD,\label{GC4_bis}
\end{enumerate}
Then there is always a solution $\Gamma^f(s)$ and any other solution $g$ is of the form 
$g(s)=e^{2\pi i a s+b} \Gamma^f(s)$ for some $a\in \ZZ$ and $b\in \CC$. If $f$ is CLD then the solutions are CLD.

Moreover, we have possible further normalizations:
\begin{itemize}
 \item  If we add the condition $g(1)=1$, or $g(k)=1$ for some $k\in \NN^*$, then all solutions are of the form $g(s)=e^{2\pi i a s} \Gamma^f(s)$.
 \item  If $f^{-1}$ has a pole at $0$ and we add the condition $\Res_{s=0} \, g=1$ then all solutions are of the form $g(s)=e^{2\pi i a s} \Gamma^f(s)$.
 \item  If $f$ has no zero at $0$ then we can add the condition $g(0)=1$ and all solutions are of the form $g(s)=e^{2\pi i a s} \Gamma^f(s)$.
 \item  If we add the conditions $g(1)=1$ and $g(\omega)\in \RR$ where $\omega \in \RR_+-\QQ$ then $g=\Gamma^f$ is unique.
 \item  If we add the conditions $g(1)=1$ and $g'(1)\in \RR$ then the solution  $g=\Gamma^f$ is unique.
 \item  If we add the hypothesis that $f$ is real analytic and the condition that $g$ is real analytic 
 then all solutions are of the form $g(s) =c .\Gamma^f (s)$ with $c\in \RR^*$.
\end{itemize}
\end{theorem}

\begin{proof}
With the same proof as before  we get the existence of a solution $\Gamma^f(s)$ such that $\Gamma^f(1)=1$.
and that any other solution is of the form $g(s)=e^{2\pi i a s+b} \Gamma^f(s)$ (note that the constant $0$ function is not LLD). 
For another solution $g$,  the condition 
$g(k)=1$ for $k\in \ZZ$ implies $e^b=1$, hence the first normalization result. For the second statement we observe that 
$$
\Res_{s=0} g= e^b \Res_{s=0} \Gamma^f 
$$
hence $e^b=1$. The third statement is similar to the first one observing that $g$ has no pole at $s=0$. The fourth normalization condition 
forces $b=0$ (first statement) and 
$$
e^{2\pi i a \omega}=1
$$
which implies $a=0$ because $\omega$ is irrational. For the fifth statement, for a second solution we have, from $g(1)=1$, $g(s)=e^{2\pi i a s} \Gamma^f(s)$. Differentiate
and set $s=1$, then  we get
$$
g'(1)=2\pi i a g(1) + \left (\Gamma^f\right )'(1) = 2 \pi i a +1 \in \RR
$$
hence $a=0$ and the solution is unique. For the last statement, $g(s) =e^{2\pi i a s+b} \Gamma^f(s)$ and $g$ and $\Gamma^f$ real analytic forces $a=0$,
and $e^b\in \RR^*$.
\end{proof}

\begin{example} \label{ex:normalizations1}
 
\normalfont{For $f(s)=s$ and the conditions $g$ real analytic and $g(1)=1$, this Theorem is just Theorem \ref{thm:main-Euler-gamma} 
and the only solution $g(s)=\Gamma (s)$ is Euler Gamma function.

Let $\omega\in \CC_+$ and consider $f(s)=\omega s$. Then $g(s)=\omega^s\Gamma(s)$ is a solution and all the  solutions are of the form 
\begin{equation*}
 g(s)=e^{2\pi i a s + s \log \omega +b}\Gamma(s) =
\end{equation*}
for  $a\in \ZZ$ and $b\in \CC$ (note that the choice of the branch of $\log \omega$ is irrelevant). 

If $\omega\in \CC^*$ and we request $g(1)=1$,  then 
% $b=-\frac{\log \omega}{\omega}$, and 
all solutions are of the form, with $a\in \ZZ$,
\begin{equation}\label{eq:normalized_solution}
g(s)=e^{(s-1 )(2\pi i a + \log \omega ) } \, \Gamma(s)
\end{equation}
If $\omega \in \RR_+$, then $f(s)=\omega s$ is real analytic, and if we request $g$ to be real analytic and $g(1)=1$, 
then, taking the real branch of $\log \omega$, we must have $a=0$ and
\begin{equation}\label{eq:normalized_solution2}
g(s)=e^{(s-1) \log \omega } \, \Gamma(s)
\end{equation}
}
\end{example}

\begin{example}\label{ex:empty_divisor}
\normalfont{
Another particular example that is worth noting in this Theorem is when $f(s)=e^{P(s)}$. Then the solutions are of the form $g(s)=e^{Q_k(s)}$ where
$$
\Delta Q_k =P +2\pi i k 
$$
for $k\in \ZZ$, where $\Delta$ is the difference operator. This means that $Q_k(s)=Q_0(s)+2\pi i k s +b$, where $b\in \CC$. 
If we want solutions normalized such that $g(1)=1$ then  $e^b=1$ and $b \in 2\pi i \ZZ$.
}
\end{example}

\subsection{A continuity result.}

We prove the continuity of the operator $\Gamma : f\mapsto \Gamma^f$ for the appropriate natural topology.

\begin{theorem}\label{thm:continuity}
Let $(f_n)_{n\geq 0}$ be a sequence of meromorphic functions with uniformly bounded convergence exponent $\alpha >0$ and such that
$$
||\Div(f_n)||_\alpha= \sum_{\rho \in \Div (f_n), \rho\not= 0} |n_\rho| |\rho|^{-\alpha} \leq M < +\infty 
$$
for a uniform bound  $M >0$. We assume that the functions $(f_n)$ 
satisfy the hypothesis of Theorem \ref{thm:general1} and 
that $f_n \to f$ when $n\to +\infty$, where $f$ is a meromorphic function  
and the convergence is uniform 
on compact sets outside the poles of $f$. Then $f$ has convergence exponent bounded by $\alpha >0$,
$$
||\Div(f)||_\alpha \leq M < +\infty 
$$
and satisfies the hypothesis of Theorem \ref{thm:general1},  
and also we have, uniformly outside the poles,
$$
\lim_{n\to +\infty} \Gamma^{f_n} =\Gamma^f
$$
\end{theorem}

\begin{proof}
We can read the divisors $\Div(f_n)$ as an integer valued functions with discrete support which 
are converging to $\Div(f)$ uniformly on compact sets. By uniform boundedness of the sums
$$
||\Div(f_n)||_\alpha= \sum_{\rho \in \Div (f_n), \rho\not= 0} |n_\rho| |\rho|^{-\alpha}
$$
we can pass to the limit and 
$$
||\Div(f)||_\alpha= \lim_{n\to +\infty} ||\Div(f_n)||_\alpha \leq M \ . 
$$
Therefore $f$ has finite order. The class of LLD real analytic functions is closed. The class of functions satisfying the functional 
equation is also closed, hence $f$ satisfies 
the hypothesis of Theorem \ref{thm:general1}, so $\Gamma^f$ is well defined.

Now, since $\Div (f_n)\to \Div (f)$, we have using Lemma \ref{lemma:divisor_determination} 
that $\Div (\Gamma^{f_n})\to \Div (\Gamma^f)$. On compact sets outside of the support of 
$\Div (\Gamma^f)$, the sequence of meromorphic functions $(\Gamma^{f_n})_{n\geq 0}$ is uniformly bounded 
(otherwise we would have a subsequence with poles out of the limit that would contradict 
the convergence of the divisor). Hence, we can extract converging subsequences. But any limit is identified 
by the uniqueness of Theorem \ref{thm:general1}, and we have convergence.
\end{proof}

\subsection{Multiplicative group property.}

Consider the space $\cE$ of LLD finite order meromorphic functions in the plane. We have that 
$$
\cE =\bigcup_{n>0} \cE_n
$$
where $\cE_n$ is the subgroup of meromorphic functions of order $\leq n$. On $\cE_n$ we consider the topology 
given by convergence of the divisor on compact sets and the convergence of functions on compact sets outside 
the limit divisor. On $\cE$ we consider the inductive topology from the exhaustion by the $\cE_n$ spaces.
Also $\cE$ and $\cE_n$ are stable under multiplication, and $(\cE, .)$ and $(\cE_n, .)$ are multiplicative 
topological group. Consider 
the closed subgroup $\cE_0\subset \cE$ of real-analytic functions $f$ normalized such that $f(1)=1$.

\begin{theorem}\label{thm:group_morphism}
The map $\Gamma :  \cE_0 \to \cE_0$ such that 
$$
\Gamma (f)=\Gamma^f
$$
is an continuous injective group morphism.
\end{theorem}

\begin{proof}Continuity results from Theorem \ref{thm:continuity}. 
We observe that from 
\begin{align*}
\Gamma^f(s+1) &= f(s) \Gamma^f(s) \\
\Gamma^g(s+1) &= g(s) \Gamma^g(s) \\
\end{align*}
we get 
$$
\Gamma^f(s+1)\Gamma^g(s+1) = f(s) g(s) \Gamma^f(s)  \Gamma^g(s)
$$
and by uniqueness of Theorem \ref{thm:general1} we get 
$$
\Gamma^f .\Gamma^g =\Gamma^{fg} \ .
$$
Also, if $\Gamma^f=1$, then directly from the functional equation we get that $f=1$, and $\text{Ker}(\Gamma)=\{1\}$.
\end{proof}

This Theorem justifies using Euler Gamma function as building block of the general solution 
by decomposing along the divisor.

\medskip

\textbf{Remark.}

\medskip

Consider the shift operator $T:\cE \to \cE$, $f(s)\mapsto T(f) =f(s+1)$ and the associated 
multiplicative cohomological equation in $g$ with 
$f$ given,
$$
T(g).g^{-1} =f \ .
$$
We have proved that the cohomological equation can be solved in $\cE$ by the group morphism $\Gamma$,  $g=\Gamma^f$. 
For $f\in \cE_\alpha$ it can be solved in $\cE_{\alpha +1}$. We observe a similar phenomenon of ``loss of regularity''
as in ``Small Divisors'' problems than in our setting can be interpreted as ``loss of transalgebraicity''.

\section{Application: Barnes higher Gamma functions.}\label{sec:Barnes_higher}

Now we can generalize the hierarchy of classical Barnes Gamma functions.

\begin{definition}\label{def:higher_gamma}
Let $f$ be a real analytic LLD meromorphic function of finite order such that $f(1)=1$.
The higher Gamma functions  associated to $f$ is a family $(\Gamma_N^f)_{N\geq 0}$ satisfying the following properties:
\begin{enumerate}
  \item $\Gamma^f_0(s)=f(s)^{-1}$,\label{HGC1}
  \item $\Gamma^f_N(1)=1$, \label{HGC2}
  \item $\Gamma^f_{N+1}(s+1)=\Gamma^f_{N}(s)^{-1} \, \Gamma^f_{N+1}(s)$, for $N\geq 0$,\label{HGC3}
  \item $\Gamma^f_N$ is a meromorphic function of finite order, \label{HGC4}
  \item $\Gamma^f_N$ is LLD,\label{HGC5}
  \item $\Gamma^f_N$ is real analytic.\label{HGC6}
\end{enumerate} 
\end{definition}

\begin{theorem}\label{thm:existence+uniqueness}
Let $f$ be a real analytic LLD meromorphic function of finite order such that $f(1)=1$. There exists 
a unique family of higher Gamma functions $(\Gamma^f_N)_N$ associated to $f$.
If $f$ is CLD then the $\Gamma^f_N$ are CLD.
\end{theorem}

\begin{proof} We set $\Gamma^f_0(s)=f(s)^{-1}$, and for $N\geq 0$, the function $\Gamma^f_{N+1}$ is 
constructed from $1/\Gamma^f_N$ using 
Theorem \ref{thm:general1}, and is unique. 
\end{proof}

The uniqueness property implies the following multiplicative group morphism property:

\begin{corollary}
For $N\geq 0$, we consider the map $\Gamma_N: \cE_0 \to \cE_0$ defined by $\Gamma_N (f) =\Gamma_N^f$. Then 
$\Gamma_N$ is a continuous injective group morphism.
\end{corollary}

\begin{proof}
Given $f,g \in \cE_0$, it is clear that the sequence of functions $\Gamma_N^f.\Gamma_N^g$ satisfy all the properties 
of higher Gamma functions associated to $fg$, hence,  by uniqueness, we have $\Gamma_N^{fg}=\Gamma_N^f\Gamma_N^g$, hence the 
group morphism property. The kernel is reduced to the constant function $1$ by uniqueness, hence the injectivity. The continuity 
follows as before from Theorem \ref{thm:continuity}.
\end{proof}

\begin{definition}[Barnes higher Gamma functons $\Gamma_N$]
The higher Gamma functions associated to $f(s)=s$ is the family of higher 
Barnes Gamma functions $(\Gamma_N)_{N\geq 0}$, where $\Gamma_1$ is 
Euler Gamma function.
\end{definition}

Note that Vign\'eras' normalization (1979, \cite{Vi}) is slightly different and defines (for $f(s)=s$) 
a hierarchy of functions
$(G_N^f)_{N\geq 0}$ as in Definition \ref{def:higher_gamma} but with the functional equation replaced by
$$
G^f_{N+1}(s+1)=G^f_{N}(s)G^f_{N+1}(s)
$$
We have a simple direct relation between the two hierarchies 
$$
G_N^f= (\Gamma_N^f)^{(-1)^{N+1}} \ .
$$ 
For $f(s)=s$ we obtain $G^f_2=G$ which is Barnes $G$-function (Barnes, 1900, \cite{Ba}). The convention in  
 Definition \ref{def:higher_gamma} is compatible with Barnes multiple Gamma functions that generalize 
 the $(\Gamma_N)$ (Barnes, 1904, \cite{Ba2}, see Section \ref{sec:multiple_gamma}).

\begin{proposition}
The higher Barnes Gamma function $\Gamma_N$ is CLD of order $N$, and 
$$
\Div(\Gamma_N) =-\sum_{n=0}^{+\infty} \binom{n+N-1}{N-1} .(-n)
$$
\end{proposition}

\begin{proof}
The function $\Gamma_N$ is in the class CLD by induction since $f$ is in this class. 
Any $\alpha >0$ is exponent of convergence for $f(s)=s$, so 
by Proposition \ref{prop:conv_exponent2} we have by induction that any $\alpha >N$ is 
exponent of convergence for $\Gamma_N$. We can check this directly using the formula 
for the divisor that follows by induction from Lemma \ref{lemma:divisor_determination} 
and the combinatorial identity
$$
\binom{n+N}{N} =\sum_{k=0}^n \binom{k+N-1}{N-1}
$$
If we write the Weierstrass factorization and $Q_N$ denotes the Weierstrass polynomial, we have 
that $\deg Q_1 =1$, and by induction the same proof gives that $\deg Q_N =N$.
\end{proof}

When we drop the real analyticity condition, there is no longer uniqueness, but we can prove the following Theorem,

\begin{theorem}\label{thm:Barnes_non_unique}
Let $f$ be a LLD meromorphic function of finite order such that $f(1)=1$.
Consider a family $(g_N^f)_{N\geq 0}$ satisfying the following properties:
\begin{enumerate}
  \item $g_0^f(s)=f(s)^{-1}$,\label{HGC1_bis}
  \item $g^f_N(1)=1$, \label{HGC2_bis}
  \item $g^f_{N+1}(s+1)=g^f_{N}(s)^{-1} g^f_{N+1}(s)$, for $N\geq 0$,\label{HGC3_bis}
  \item $g^f_N$ is a meromorphic function of finite order, \label{HGC4_bis}
  \item $g^f_N$ is LLD,\label{HGC5_bis}
\end{enumerate} 
Then there exists an integer sequence $(a_k)_{k\geq 0}$,  such that  
$$
g^f_N(s) =\exp \left (  2 \pi i \sum_{k=0}^N a_{N-k} \binom{s}{k}  \right ) \Gamma^f_N(s)
$$
\end{theorem}

\begin{proof}
This follows by induction from Theorem \ref{thm:general1_bis}. We can also give a direct argument using the group structure. 
For any solution $(g_N^f)_{N\geq 0}$, the functions $h_N^f =\Gamma^f_N /g^f_N$ are solution for $f=1$. The case $f=1$ is easily resolved. 
By induction, the solutions have no zeros nor poles, and finite order, so 
we have 
$$
h_N^f (s) =e^{2 \pi i A_N(s)}
$$
where the $(A_N)_{N\geq 0}$ is a sequence of polynomials satisfying
$$
\Delta A_{N+1} =-A_N
$$
and $A_0(s)= a_0 \in \ZZ$. The difference equation and the sequence $a_N=(-1)^N A_N(0)$ determines the sequence of polynomials $(A_N)_{N\geq 0}$ that are given 
by the explicit formula
$$
A_N(s)= \sum_{k=0}^N a_{N-k} \binom{s}{k}
$$
\end{proof}

\section{Application: Jackson $q$-Gamma function.}

For $0<q<1$, Jackson (1905, \cite{Ja1}, \cite{Ja2}) (see also the precursor work by Halphen \cite{Ha}, 
vol. 1, p. 240; and H\"older \cite{Ho})  defined the $q$-Gamma function $\Gamma_q$ by the product formula
$$
\Gamma_q(s) =\frac{(q;q)_\infty}{(q^s;q)_\infty} (1-q)^{1-s}
$$
where the $\infty$-Pochhammer symbol is 
$$
(z;q)_\infty=\prod_{k=0}^{+\infty} (1-zq^k) \ .
$$
The $q$-Gamma function satisfies the functional equation
$$
\Gamma_q(s+1) = \frac{1-q^s}{1-q} \, \Gamma_q(s)
$$
and Euler Gamma function appears as the limit when $q\to 1$,
$$
\Gamma (s)=\lim_{q\to 1-0} \Gamma_q(s)
$$

Askey (\cite{As}, 1980) proved a $q$-analog of the Bohr-Mollerup theorem characterizing 
$\Gamma_q$ by its functional equation, the normalization $\Gamma_q(1)=1$, 
and the real log-convexity of $\Gamma_q$. It is natural to investigate if we can use our approach. 
The answer is affirmative as shows the next Theorem.

\begin{theorem}
The $q$-Gamma function is the only real analytic, finite order meromorphic function such that $\Gamma_q(1)=1$ and 
satisfying the functional equation,
$$
\Gamma_q(s+1) = \frac{1-q^s}{1-q} \, \Gamma_q(s)
$$
\end{theorem}

\begin{proof}
This is an application of our general Theorem \ref{thm:general1} with 
$$
f(s) =\frac{1-q^s}{1-q}
$$
which is an order $1$ real analytic function in  the class LLD (but not CLD), $f(1)=1$, and 
$$
\Div (f) = \sum_{k\in \ZZ} 1. \left (\frac{2 \pi i k}{\log q} \right )\ .
$$
\end{proof}

An application of the continuity Theorem \ref{thm:continuity} shows:

\begin{proposition}
We have
$$
\lim_{q\to 1-0} \Gamma_q =\Gamma
$$
uniformly on compact sets of $\CC$.
\end{proposition}

\begin{proof}
Uniformly on compact sets of $\CC$ we have
$$
\lim_{q\to 1-0} \frac{1-q^s}{1-q} =s
$$
and we use Theorem \ref{thm:continuity}.
\end{proof}

\medskip

Nishiwaza (1996, \cite{Ni1996}) has defined the $q$-analog $\Gamma_{N, q}$ of Barnes higher 
Gamma functions $\Gamma_N$ following the Bohr-Mollerup approach. With our methods we can obtain 
Nishiwaza's $\Gamma_{N, q}$ functions directly from the higher hierarchy generated 
by $f$ using Definition \ref{def:higher_gamma}
and Theorem \ref{thm:existence+uniqueness} using the uniqueness of the solution.

\begin{theorem}
Nishiwaza's higher $q$-Gamma functions $\Gamma_{N,q}$ are obtained  by the higher hierarchy from Theorem \ref{thm:existence+uniqueness}
$$
\Gamma_{N,q} = \Gamma_N^f
$$
associated to the real analytic function
$$
f(s)=\frac{1-q^s}{1-q} \ .
$$
\end{theorem}

\section{Application: Mellin Gamma functions.}\label{sec:Mellin_gamma}

Mellin (1897, \cite{Me1897}) considered general Gamma functions satisfying the functional equation
$$
F(s+1) = R(s) F(s)
$$
where $R$ is a rational function. He constructs solutions by using Euler Gamma function as 
building block along the divisor. An application of the extension of our 
general Theorem \ref{thm:general1_bis}, and the group structure Theorem \ref{thm:group_morphism},  gives the
precise existence characterization of Mellin Gamma functions. 

\begin{definition}
A meromorphic function $f$ is LLD at infinite if $f(s+a)$ is LLD for some $a\in\RR$. 
\end{definition}

Since $\Div(f(s+a))=\Div(f) -a$ this means that the divisor of $f$ is in some left half plane (not necessarily $\CC_+$).

\begin{theorem}
Let $R$ be a rational function,
$$
R(s) =a\frac{(s-\alpha_1)\ldots (s-\alpha_n)}{(s-\beta_1)\ldots (s-\beta_m)}
$$
where $a\in \CC^*$, and $(\alpha_k)$ and $(\beta_k)$ are the zeros, resp. the poles, of $R$ counted with multiplicity.

Consider the finite order meromorphic functions, LLD at infinite,  that are solutions of the functional equation
\begin{equation}\label{eq:Mellin_functional}
F(s+1) = R(s) F(s) \ . 
\end{equation}

They are of the form
$$
F(s)=a^s\frac{\Gamma(s-\alpha_1)\ldots \Gamma(s-\alpha_n)}{\Gamma(s-\beta_1)\ldots \Gamma(s-\beta_m)} e^{2 \pi i k s}
$$
for some $k\in \ZZ$.

In particular, if $R(1)=1$ and $R$ is real analytic there is only one real analytic solution such that $F(1)=1$.

\end{theorem}

\begin{proof}
Let   $\alpha$ be a zero or pole. We consider the linear function $f_{\alpha}(s)=s-\alpha$ 
and a solution $\Gamma^{f_{\alpha}}$ to
$$
F_{\alpha}(s+1)= f_{\alpha}(s) F_{\alpha}(s+1) \ .
$$
Also $a^s$ is a solution to $F(s+1)=a F(s)$. 
Then, Theorem \ref{thm:general1_bis} and the group structure of the solutions, Theorem \ref{thm:group_morphism}, 
shows that the general solutions of the functional equation (\ref{eq:Mellin_functional})
are of the form
\begin{align*}
F(s) &=a^se^{2 \pi i n s}\frac{\Gamma^{f_{\alpha_1}}(s)e^{2 \pi i k_1 s}\ldots 
\Gamma^{f_{\alpha_n}}(s) e^{2 \pi i k_n s}}{\Gamma^{f_{\beta_1}}(s) e^{2 \pi i l_1 s}\ldots \Gamma^{f_{\beta_m}}(s)e^{2 \pi i l_m s}}\\
&=a^s\frac{\Gamma^{f_{\alpha_1}}(s)\ldots 
\Gamma^{f_{\alpha_n}}(s)}{\Gamma^{f_{\beta_1}}(s) \ldots \Gamma^{f_{\beta_m}}(s)}e^{2 \pi i k s}
\end{align*}
where $n, k_1,\ldots , k_n, l_1, \ldots , l_m \in \ZZ$, and $k=n+k_1+\ldots k_n+l_1+\ldots +l_m$.

We finish the proof by observing that we can take $\Gamma^{f_{\alpha}}(s) =\Gamma (s-\alpha)$.
When $R$ is real analytic, $a\in \RR^*$, the set of roots $(\a_j)$ and poles $(\b_j)$ are self-conjugated, and we must have $k=0$ to 
have $F$ real analytic.
\end{proof}

Considering a LLD rational function $R$, real analytic and such that $R(1)=1$, we can define the unique associated higher Gamma functions
hierarchy $(\Gamma_N^R)_{N\geq 0}$ given by Theorem \ref{thm:existence+uniqueness}. These higher Mellin Gamma functions seem to be new in the
literature.

\section{Application: Barnes multiple Gamma functions.}\label{sec:multiple_gamma}

For $N\geq 1$ and parameters $\boldsymbol{\omega}=(\omega_1, \ldots , \omega_n) \in \CC_+^n$,
Barnes multiple Gamma functions $\Gamma(s|\omega_1,\ldots ,\omega_N)=\Gamma(s|\boldsymbol{\omega})$ 
are a generalization by Barnes (1904, \cite{Ba2}) of Barnes higher Gamma functions $\Gamma_N$
studied in section \ref{sec:Barnes_higher}.  When 
$\omega_1=\ldots =\omega_N=1$ we recover $\Gamma_N$ as
$$
\Gamma_N (s)= \Gamma(s|1,\ldots ,1)
$$
Barnes only considers the apparently more general case
where $\omega_1, \ldots , \omega_n$ all belong to a half plane limited by a line through the origin (\cite{Ba2} p.387).
This situation that can be reduced to our case by a rotation.
Also, he assumes $\dim_\QQ (\omega_1, \ldots , \omega_N)\geq 3$ to have an essentially different situation 
from the double Gamma function $G$ that he studied previously, although this condition is not the appropriate one.
Barnes defines these multiple Gamma functions \textit{\`a la Lerch}.  
First, Barnes defines the Barnes-Hurwitz zeta functions, a multiple version of Hurwitz zeta function, as
$$
\zeta(t, s|\omega_1,\ldots ,\omega_N) = \sum_{k_1,\ldots k_N\geq 0} (s+k_1\omega_1 +\ldots +k_N\omega_N)^{-t} \ ,
$$
which is converging for $\Re s > N$, and symmetric on $\omega_1,\ldots ,\omega_N$. This multiple zeta function reduces to 
Hurwitz zeta function for $N=1$ (Hurwitz, 1882, \cite{Hurwitz1882}). Its analytic continuation and Lerch formula 
(Lerch, 1894, \cite{Le1894})
\begin{equation}\label{eqn:Lerch_formula2}
\log \Gamma(s) = \left [\frac{\partial}{\partial t}\zeta(t,s) \right ]_{t=0}-\zeta'(0) 
\end{equation}
allows to define Euler Gamma function. Barnes generalizes this approach and 
he shows, using a Hankel type integral, that $\zeta(s, t|\omega_1,\ldots ,\omega_N)$ has a meromorphic extension in $(s,t)$. Then he 
defines
$$
\Gamma_B(s|\boldsymbol{\omega}) = \rho_N (\boldsymbol{\omega}) \exp \left (\left [\frac{\partial}{\partial t} \zeta(t,s|\boldsymbol{\omega}) \right ]_{t=0} \right )
$$
where $\rho_N (\boldsymbol{\omega})$ is Barnes modular function, and is defined to provide  the normalization 
such that $\Gamma_B(s|\boldsymbol{\omega})$ has residue $1$ at $s=0$,

\begin{equation}\label{eq:residue_normalization}
\Res_{s=0} \Gamma_B(s|\boldsymbol{\omega}) =\lim_{s\to 1} s \Gamma_B(s|\boldsymbol{\omega}) =1 
\end{equation}
From the definition we get that both $\rho_N (\boldsymbol{\omega})$ and $\Gamma_B(s|\boldsymbol{\omega})$ are 
necessarily symmetric on $\omega_1,\ldots ,\omega_N$.
Note that for Euler Gamma function, because of the form of the functional equation, the normalization 
$\Gamma (1)=1$ is equivalent to $\Res_{s=0}\Gamma =1$. In general, for $\Gamma^f$ the normalization $\Gamma^f(1)=1$ is equivalent 
to 
$$
\Res_{s=0} \Gamma^f = \Res_{s=0} f^{-1} \ .
$$
For Barnes higher Gamma functions $\Gamma_N$ discussed in section \ref{sec:Barnes_higher}, we see that the normalization $\Gamma_N(1)=1$ is 
equivalent to $\Res_{s=0} \Gamma_N =1$ when we make $s\to 0$ in
$$
\Gamma_{N+1}(s+1) = \left (s \Gamma_N(s)\right )^{-1} \, s\Gamma_{N+1} (s)
$$
we get
$$
\Gamma_{N+1}(1) = \Res_{s=0} \Gamma_{N+1} =1 \ . 
$$
and the result follows by induction.

Barnes (\cite{Ba2}, p.397) observes that $\log \rho (\boldsymbol{\omega})$ plays the role  
of Stirling's constant of the asymptotic expansion when $k\to +\infty$ of the divergent sum
$$
\sum_{\omega\in \Omega^*, |\omega|\leq k} \log |\omega| 
$$
where $\Omega^* =\NN . \omega_1 +\NN . \omega_2+\ldots +\NN . \omega_N -\{0\}$. Also, $\log \rho (\boldsymbol{\omega})$ can be defined in this way.

Later applications to Number Theory by Shintani in the 70's of Barnes multiple Gamma functions (1976,\cite{Shi1}, \cite{Shi2}, \cite{Shi3}), 
and modern presentations (Ruijsenaars, \cite{Rui2000}), 
drop Barnes normalization. They define multiple Gamma functions directly by the formula
$$
\Gamma (s|\boldsymbol{\omega}) = \exp \left (\left [\frac{\partial}{\partial t} \zeta(t,s|\boldsymbol{\omega}) \right ]_{t=0} \right )
$$
We keep Shintani's normalization that has become the standard one in the modern literature. 
This normalization has  the advantage to yield a simpler functional 
equation not involving Barnes modular function $\rho (\boldsymbol{\omega})$. 
We denote by $\boldsymbol{\hat \omega}$ the $N-1$ dimensional vector obtained from 
$\boldsymbol{\omega}$ removing the $k$-th coordinate. Then we have the following ladder functional equation for the zeta function,
\begin{equation}\label{eq:lader_eq_multiple_zeta}
\zeta (t, s+\omega_k|\boldsymbol{\omega})- \zeta(t, s|\boldsymbol{\omega}) = -\zeta(t, s|\boldsymbol{\hat \omega})
\end{equation}
where we start with
$$
\zeta(t,s|\emptyset)=s^{-t} \ .
$$
From the zeta function functional equation we get the functional equation for the multiple Gamma functions,
\begin{equation}\label{eq:functional_eq_multiple_gamma}
\Gamma (s+\omega_k|\boldsymbol{\omega})= \Gamma(s|\boldsymbol{\hat \omega})^{-1}\Gamma(s|\boldsymbol{\omega})  
\end{equation}
with the convention $\Gamma(s|\emptyset)=s$.
Note that the functional equation for Barnes normalized multiple Gamma functions is different, and we have a different functional equation:
\begin{equation}\label{eq:Barnes_functional_eq_multiple_gamma}
\Gamma_B (s+\omega_k|\boldsymbol{\omega})= \rho (\bld{\hat \omega}) \Gamma_B(s|\boldsymbol{\hat \omega})^{-1}\Gamma_B(s|\boldsymbol{\omega})  \ .
\end{equation}

\begin{example}

\normalfont{
For $N=1$, $\Gamma(s|\omega)$ can be computed explicitly from Euler Gamma function (see \cite{Shi0}, p.203).
}
 
\end{example}

\begin{lemma}\label{lemma:formulas_N=1} 
We have
\begin{align*}
\Gamma(s|\omega) &=(2\pi)^{-1/2} e^{\left (\frac{s}{\omega}-\frac12 \right )\log \omega} \, \Gamma\left (\frac{s}{\omega}\right )\\
\rho_1(\omega) &=\sqrt{\frac{\omega}{2\pi}}
\end{align*}
and therefore
$$
\Gamma_B(s|\omega) = \sqrt{\frac{2 \pi}{\omega}}\, \Gamma (s|\omega)= e^{\left (\frac{s}{\omega}-1 \right )\log \omega} \, \Gamma\left (\frac{s}{\omega}\right )
$$
and
\begin{align*}
\Gamma(\omega|\omega) &=  \sqrt{\frac{\omega}{2\pi}} \\
\Res_{s=0} \Gamma(s|\omega) &= \sqrt{\frac{\omega}{2\pi}} \\
\Gamma_B(\omega|\omega) &= 1 \\
\Res_{s=0} \Gamma_B(s|\omega) &= 1
\end{align*}

In particular,
\begin{align*}
\Gamma(s|1) &= \frac{\Gamma (s)}{\sqrt{2\pi}} \\
\Gamma_B(s|1) &= \Gamma (s) 
\end{align*}

\end{lemma}

\begin{proof}
For $\omega=1$, $\zeta(t,s|1)=\zeta(t,s)$ is the original Hurwitz zeta function that 
generalizes Riemann zeta function $\zeta(t)=\zeta(t,1)$,
$$
\zeta(t,s)=\sum_{k\geq 0} (s+k)^{-t}  \ .
$$
Making $t=0$ in the first formula from  Lemma 3.18 from \cite{PM1} we have the classical result (see also \cite{WW} p.267)
\begin{equation}\label{Hurwitz_t=0}
\zeta(0,s)=\frac12 -s \ .
\end{equation}
Observe now that we have 
$\zeta(t,s|\omega) = \omega^{-t} \zeta\left (t,\frac{s}{\omega}\right )$,
hence 
$$
\frac{\partial}{\partial t} \zeta(t,s|\omega) =-(\log \omega) \omega^{-t} \zeta\left (t, \frac{s}{\omega}\right ) +\omega^{-t} \frac{\partial}{\partial t} \zeta(t,s)
$$
and making $t=0$, using formula (\ref{Hurwitz_t=0}) and  Lerch formula (\ref{eqn:Lerch_formula2}), we get
$$
\log \Gamma (s|\omega) =  \left (\frac{s}{\omega}-\frac12 \right ) \log \omega +\log \Gamma\left (\frac{s}{\omega}\right ) +\zeta'(0)  \ .
$$
Now, $\zeta'(0)=-\frac12 \log (2\pi)$ gives the first formula. Then using this formula we get
$$
\Res_{s=0} \Gamma(s|\omega) = \lim_{s\to 1} s\Gamma(s|\omega) = (2\pi)^{-1/2} e^{-\frac12 \log \omega} \omega= \sqrt{\frac{\omega}{2\pi}} \ .
$$
\end{proof}

For $N\geq 2$ we create new transcendentals $\Gamma(s|\bld \omega )$, which are not generated from Euler Gamma function.  
For example for $N=2$, if $\omega_1$ and $\omega_2$ are $\QQ$-independent we get new transcendentals. When the parameters are $\QQ$-dependent 
then $\Gamma(s|\omega_1 \omega_2)$ can 
be expressed from  Barnes $G$-function, $G_2=\Gamma_2^{-1}$.

Considering the functional equations, from our point of view, it is natural to aim to characterize $\Gamma (s|\boldsymbol{\omega})$ by solving a tower of 
difference equations corresponding to the sequence $(\omega_k)_{1\leq k\leq n}$. Our approach leads to a new definition, not needing 
Barnes-Hurwitz zeta functions.
We start by considering real analytic 
multiple zeta functions that are those relevant in Shintani's applications to real quadratic number fields (1978, \cite{Shi2}).
The following Theorem follows from Theorem \ref{thm:general1_bis}.

\begin{theorem}\label{thm:general2}
Let $\omega\in \RR_+$.
Let $f$ be a real analytic LLD meromorphic function in $\CC$ of finite order. 
There exists a unique function $\Gamma^f(s|\omega)$  satisfying the following properties:
\begin{enumerate}
  \item $\Gamma^f(1|\omega)=1$ , \label{MGC1}
  \item $\Gamma^f(s+\omega|\omega)=f(s)\Gamma^f(s|\omega)$ ,\label{MGC2}
  \item $\Gamma^f(s|\omega)$ is a meromorphic function of finite order, \label{MGC3}
  \item $\Gamma^f(s|\omega)$ is LLD,\label{MGC4}
  \item $\Gamma^f$ is real analytic.\label{MGC5}
\end{enumerate} 

If $f$ is CLD then $\Gamma^f$ is CLD.

If we drop condition (\ref{MGC1}) then $\Gamma^f(s|\omega)$ is unique up 
to multiplication by a constant $c\in \RR^*$.

If $\Res_{s=0} f^{-1}=1$, we can replace condition (\ref{MGC1}) by the condition $\Res_{s=0} \Gamma^{f}=1$.
\end{theorem}

\begin{proof}
We make the change of variables $t=\omega^{-1}s$. The application of Theorem \ref{thm:general1_bis}  to 
the real analytic function $h(t)=f(\omega t)$ gives a unique real analytic solution $\Gamma^h(t)$ such that $\Gamma^h(1)=1$ and
$$
\Gamma^h(t+1)=h(t) \Gamma^h(t) \ .
$$
If we set $\Gamma^f(s|\omega) = \Gamma^h(\omega^{-1} s)$, this equation becomes
$$
\Gamma^f(s+\omega|\omega)= \Gamma^h (\omega^{-1} s +1) = h(\omega^{-1} s) \Gamma^h(\omega^{-1} s)=f(s)\Gamma^f(s|\omega)
$$
and $\Gamma^f(s|\omega)$ satisfies all conditions. 
Furthermore, $\Gamma^f(s|\omega)$ is unique from the uniqueness of $\Gamma^h$ that follows from 
the last uniqueness condition in Theorem \ref{thm:general1_bis}. In view of this uniqueness 
result, the two last statement are clear. Also if $f$ is CDL then $\Gamma^f(s|\omega)$ is CDL.
\end{proof}

\begin{example}

\normalfont{
For $f(s)=s$ the proof gives $h(t)=\omega t$ and a solution $\Gamma^f(s|\omega)=\Gamma^h \left (\frac{t}{\omega}\right )$. 
The condition  $\Gamma^f(1|\omega)=1$ 
is equivalent to $\Gamma^h \left (\frac{1}{\omega}\right )=1$, then according to Example \ref{ex:normalizations1} there is a 
unique real analytic solution
$$
\Gamma^h (t) = e^{(t-1)\log \omega} \, \frac{\Gamma (t)}{\Gamma(\omega^{-1})}
$$
and it follows that 
$$
\Gamma^f(s|\omega) = e^{\left (\frac{s}{\omega}-1\right ) \log \omega} \, 
\frac{\Gamma\left ( \frac{s}{\omega} \right )}{\Gamma(\omega^{-1})} 
$$
Therefore, by uniqueness of the normalization, 
$$
\Gamma_B(s|\omega) =\Gamma(\omega^{-1})\Gamma^f(s|\omega)
$$
and we recover the formula for $\Gamma_B(s|\omega)$ from Lemma \ref{lemma:formulas_N=1}
$$
\Gamma_B(s|\omega) = e^{\left (\frac{s}{\omega}-1\right ) \log \omega} \, \Gamma\left ( \frac{s}{\omega} \right ) 
$$
Then the formula for  $\Gamma (s|\omega)$ follows from 
$$
\Gamma(s|\omega) =\sqrt{\frac{\omega}{2\pi}} \, \Gamma_B(s|\omega) = (2\pi)^{-1/2} 
e^{\left (\frac{s}{\omega}-\frac12 \right )\log \omega} \, \Gamma\left (\frac{s}{\omega}\right )\ . 
$$

\medskip

We have established,
\begin{proposition}
For $f(s)=s$ we have 
$$
\Gamma(s|\omega) =\sqrt{\frac{\omega}{2\pi}}\, \Gamma(\omega^{-1})\, \Gamma^f(s|\omega)
$$
where $\Gamma^f(s|\omega)$ is the unique solution in Theorem \ref{thm:general2}.
\end{proposition}
}
\end{example}

Using similar ideas, the general version of Theorem \ref{thm:general1_bis} for $\omega\in \CC_+$ and 
without the hypothesis of $f$ being real analytic is the following:

\begin{theorem}\label{thm:general2_bis}
Let $\omega\in \CC_+$.
Let $f$ be a LLD meromorphic function in $\CC$ of finite order. 
We consider a function $g$ satisfying
\begin{enumerate}
  \item $g(1)=1$, \label{GC1_bis2}
  \item $g(s+\omega)=f(s)g(s)$,\label{GC2_bis2}
  \item $g$ is a meromorphic function of finite order, \label{GC3_bis2}
  \item $g$ is LLD,\label{GC4_bis2}
\end{enumerate}
Then there is a solution $\Gamma^f(s|\omega)$. Any other solution $g$ is of 
the form $g(s)=e^{2\pi i a \frac{s-1}{\omega}} \, \Gamma^f(s|\omega)$ for some $a\in \ZZ$. 

\noindent If we remove condition (1) then all solutions are of the form 
the form $g(s)=e^{b+2\pi i a \frac{s}{\omega}} \, \Gamma^f(s|\omega)$ for some $a\in \ZZ$ and $b\in \CC$. 
\end{theorem}

\begin{proof}As before, we make the change of variables $t=\omega^{-1}s$ and apply 
Theorem \ref{thm:general1_bis}  to 
the function $h(t)=f(\omega t)$ gives an unconditional solution 
$\Gamma^f(s|\omega)=\Gamma^h(t)/\Gamma^h(\omega^{-1})$. From the general uniqueness statement 
in \ref{thm:general1_bis} we know that 
all the other solutions removing condition (1) are of the form 
$g(s)=e^{2\pi i a \frac{s}{\omega}+b} \, \Gamma^f(s|\omega)$ for some $a\in \ZZ$ and $b\in \CC$. Condition (1)
is then equivalent to $2\pi i a/\omega+b = 2\pi i k$ with $k\in  \ZZ$, hence the general form.
\end{proof}

Therefore, in general for $\omega \in \CC^*$, $\Gamma^f$ is not uniquely determined, but its values on $1+\ZZ . \omega$
are well determined. More precisely, we have,

\begin{proposition}\label{prop:unique}
The values taken by solutions at the points $1+k\omega$ for $k\in \ZZ$ 
are uniquely determined and do not depend on the solution chosen. 

If $\omega \in \CC_+$, any solution $g$ is uniquely determined by $\Im g'(1)$, in particular, if $f$ is real analytic then 
there is a unique real analytic solution.

If $\dim_\QQ (1, \omega) =2$, any solution $g$ is uniquely determined by its value $g(r)$ 
for rational value $r\not=1$ and $r>0$.
\end{proposition}

\begin{proof}
From the functional equation we have
$$
g(1+k\omega)=g(1) \prod_{j=0}^{k-1} f(1+j\omega) = \prod_{j=0}^{k-1} f(1+j\omega)
$$
hence the first claim.

Now, consider two solutions $g_1$ and $g_2$ such that $\Im g_1'(1) = \Im g_2'(1)$. 
Since they are of the form $g_j(s)=e^{2\pi i a_j \frac{s-1}{\omega}} \, \Gamma^f(s|\omega)$ 
for some $a_j\in \ZZ$, taking logarithmic derivatives we have
$$
g_j'(1) =\frac{g_j'(1)}{g_j(1)} = 2\pi i \frac{a_j}{\omega} +\frac{\left (\Gamma^f\right )'(1|\omega)}{\Gamma^f(1|\omega)}
= 2\pi i \frac{a_j}{\omega} +\left (\Gamma^f\right )'(1|\omega)
$$
hence
$$
g_1'(1) - g_2'(1) =2\pi i \, \frac{a_1-a_2}{\omega} \in \RR
$$
and the condition $\omega \in \CC_+$  forces $a_1=a_2$.

Now assume $\dim_\QQ (1, \omega) =2$ and consider two solutions $g_1$ and $g_2$ 
such that $g_1(r) = g_2 (r)$ for some rational $r\in \QQ$ with $r>0$ and $r\not=1$. Then, since $s=r$ is neither a zero nor a pole
because $g$ is LLD, we have
for some $l\in \ZZ$
$$
\frac{g_1(r)}{g_2(r)} = e^{2\pi  (r-1) \frac{a_1-a_2}{\omega}} =1
$$
thus we have
$$
(r-1)(a_1-a_2)-l\omega =0
$$
and, by $\QQ$-independence, we must have $l=0$ and $(r-1)(a_1-a_2)=0$, thus, since $r\not= 1$, we must have $a_1=a_2$ and $g_1=g_2$.
\end{proof}

\medskip
\textbf{General Multiple Gamma Hierarchies.}
\medskip

Now, we can iterate Theorem \ref{thm:general2} to define new real-analytic multiple 
Gamma functions corresponding to $f$ and to positive real parameters 
$\boldsymbol{\omega} =(\omega_1, \ldots, \omega_N) \in \RR_+^N$

For an infinite sequence of parameters $\bld \omega=(\omega_1, \omega_2, \ldots )\in \CC^\infty_+$, we can 
also define a generalization of  Barnes multiple Gamma hierarchy. We denote $\bld \omega_N=(\omega_1, \ldots, \omega_N)\in \CC^N_+$.

\begin{definition}[General Multiple Gamma Hierarchy] \label{def:General_Multiple_Gamma_Hierarchy}
Let $\bld \omega=(\omega_1, \omega_2, \ldots )\in \CC^\infty_+$ and $f$ be a LLD meromorphic function 
in $\CC$ of finite order. A general multiple Gamma hierarchy 
$(\Gamma^{f}_N(s|\boldsymbol{\omega}_N))_{N\geq 0}$
associated to $f$ is a sequence of functions satisfying:
\begin{enumerate}
  \item $\Gamma^f_0(s)=f(s)^{-1}$,\label{MHGC1}
  \item $\Gamma^{f}_{N+1}(s+\omega_{N+1}|\boldsymbol{\omega}_{N+1})=
  \Gamma^{f}_N (s|\boldsymbol{\omega}_{N})^{-1} \, \Gamma_{N+1}^{f}(s|\boldsymbol{\omega}_{N+1})$, 
  for $N\geq 0$,\label{MHGC3}
  \item $\Gamma^{f}_N(s|\boldsymbol{\omega}_N) $ is a meromorphic function of finite order, \label{MHGC4}
  \item $\Gamma^{f}_N(s|\boldsymbol{\omega}_N) $ is LLD.\label{MHGC5}
\end{enumerate} 
\end{definition}
Next we show that, with some simple normalization, General Multiple Gamma Hierarchies are unique for real parameters and $f$ real analytic.

\begin{theorem}\label{thm:existence+uniqueness_multiple_f(1)=1} 
Let $\bld \omega=(\omega_1, \omega_2, \ldots )\in \RR^\infty_+$ and 
$f$ a real analytic LLD meromorphic function of finite order, such that $f(1) =1$.
There exists a unique General Multiple Gamma Hierarchy 
$(\Gamma^{f}_N(s|\boldsymbol{\omega}_N))_{N\geq 1}$ associated to $f$, 
and normalized such that 
$$
\Gamma^{f}_N(1|\boldsymbol{\omega}_N) =1 \ .
$$
If $f$ is CLD then the $\Gamma^{f}_N(s|\boldsymbol{\omega}_N)$ are CLD.
\end{theorem}

\begin{proof}
The existence and uniqueness is proved  by induction on $N\geq 0$. For $N=0$, $\Gamma^f_0(s)=f(s)$.
We assume that the result has been proved for $N\geq 0$. Then we construct 
$\Gamma^{f}_{N+1}(s|\boldsymbol{\omega}_{N+1})$ by using Theorem \ref{thm:general2} using the function 
$f=\Gamma^{f}_{N}(s|\boldsymbol{\omega}_N)^{-1}$. 
\end{proof}

In the particular case where $f(s)=s$, using uniqueness, we obtain Barnes Multiple Gamma functions for real parameters $\boldsymbol{\omega}$. 
This gives a new approach to define Barnes Multiple Gamma functions.

\begin{definition}[Barnes Multiple Gamma functions]
For $\bld \omega=(\omega_1, \omega_2, \ldots )\in \RR^\infty_+$ the General Multiple Gamma Hierarchy associated to $f(s)=s$ is 
Barnes Multiple Gamma Hierarchy $(\Gamma^f_N (s|\boldsymbol{\omega}_N))_{N\geq 1}$ with the normalization $\Gamma^f_N (1|\boldsymbol{\omega}_N)=1$.
We simplify the notation and we denote $\Gamma^f_N (s|\boldsymbol{\omega}_N)=\Gamma (s|\boldsymbol{\omega}_N)$.
\end{definition}

We observe that since the Barnes multiple Gamma functions $\Gamma(s|\boldsymbol{\omega}_N)$ are symmetric on the 
real parameters $(\omega_1, \ldots , \omega_N)$ then, by uniqueness, the solutions of 
Theorem \ref{thm:existence+uniqueness_multiple_f(1)=1} 
for $f(s)=s$ must also be symmetric on the parameters, which is not obvious a priori. 
This is general when we can define the Gamma functions \`a la Lerch,  including the case of complex 
parameters $\bld \omega=(\omega_1, \omega_2, \ldots )\in \CC^\infty_+$. 
Consider $f$ a real analytic LLD meromorphic function of finite order, such that,
$$
f(1) =1
$$
and $\Re f(s) >0$ for $s\in \CC_+$. These conditions are sufficient to define $f(s)^{-t}$ for $s\in \CC_+$ by taking
the principal branch of $\log $ in $\CC_+$, $f(s)^{-t}=\exp (-t \log f(s))$. 
We assume that the  multiple Barnes-Hurwitz multiple zeta function associated to $f$,
$$
\zeta^f(t, s|\omega_1,\ldots ,\omega_N) = \sum_{k_1,\ldots k_N\geq 0} f(s+k_1\omega_1 +\ldots +k_N\omega_N)^{-t} \ ,
$$
is well defined and holomorphic in a right half plane $\Re t > t_0$ for all $s\in \CC_+$, and has 
a meromorphic extension to $t\in \CC$. Then we can define $\Gamma^f (s|\emptyset)=f(s)^{-t}$, and, 
\`a la Lerch, for $s\in \CC_+$, 
$$
\Gamma^f_L (s|\boldsymbol{\omega}_N) = \exp \left (\left [\frac{\partial}{\partial t} \zeta^f(t,s|\boldsymbol{\omega}_N) \right ]_{t=0} - 
\left [\frac{\partial}{\partial t} \zeta^f(t,s|\boldsymbol{\omega}_N) \right ]_{t=0, s=1} \right ) 
$$
Note that we have normalized these functions such that $\Gamma^f_L (1|\boldsymbol{\omega}_N) =1$.  By construction, these functions are 
obviously symmetric on the parameters $\omega_1, \ldots , \omega_N$.
As before, these functions satisfy the functional equations,
\begin{equation}\label{eq:functional_eq_general_multiple_gamma}
\Gamma^f_L (s+\omega_N|\boldsymbol{\omega}_N)= \Gamma^f_L(s|\boldsymbol{\omega}_{N-1})^{-1}\Gamma^f_L(s|\boldsymbol{\omega}_N)  
\end{equation}
which show that they have a meromorphic extension to all $s\in \CC$. now, using the uniqueness from 
Theorem \ref{thm:existence+uniqueness_multiple_f(1)=1} we get for real parameters:

% \medskip

\begin{theorem}
Let $\boldsymbol{\omega}=(\omega_1, \omega_2, \ldots )\in \RR^\infty_+$.
When $\Gamma^f_L (s|\boldsymbol{\omega}_N)$ is well defined, we have 
$$
\Gamma^f_L (s|\boldsymbol{\omega}_N)=\Gamma^f_N (s|\boldsymbol{\omega}_N)
$$
where the $(\Gamma^f_N (s|\boldsymbol{\omega}_N))_{N\geq 0}$ are the solutions of Theorem \ref{thm:existence+uniqueness_multiple_f(1)=1}.
\end{theorem}

\begin{corollary}Let $\boldsymbol{\omega}=(\omega_1, \omega_2, \ldots )\in \RR^\infty_+$.
The Barnes multiple Gamma hierarchy defined by Theorem \ref{thm:existence+uniqueness_multiple_f(1)=1}, 
$\Gamma^f_N (s|\boldsymbol{\omega}_N)$ are symmetric on the parameters $\boldsymbol{\omega}_N=(\omega_1, \omega_2, \ldots,\omega_N)$. 
\end{corollary}

We note that our definition of the hierarchies using the functional equation 
is more general than Barne's definition \`a la Lerch, since we need extra conditions 
on $f$ to define the multiple $f$-Barnes-Hurwitz zeta function and prove that it is holomorphic in a half 
plane. If we don't add the normalization condition
$$
\Gamma^{f}_N(1|\boldsymbol{\omega}_N) =1 
$$
then there exists solutions that are non-symmetric on the parameters.
As we see next, this is even more evident for complex parameters since in that case, without further 
hypothesis, there is no symmetry on the parameters $\boldsymbol{\omega}$. This proves that our functional equation approach 
defines a larger class of functions than the classical ones.

\medskip

We observe also that the existence and uniqueness of Theorem \ref{thm:existence+uniqueness_multiple_f(1)=1} implies the 
morphism property. Let $\cE^\RR$ be the multiplicative group of real-analytic LLD meromorphic functions of finite order and 
$$
\cE^\RR =\bigcup_{n\geq 1} \cE^\RR_n
$$ 
and let $\cE^\RR_0$ the subgroup of functions $f$ such that $f(1)=1$. With the same arguments as before, we prove:
\begin{theorem}
 For $\boldsymbol{\omega}=(\omega_1, \omega_2, \ldots )\in \RR^\infty_+$ and $N\geq 0$, we consider the map 
$$
\Gamma_N ( \boldsymbol{\omega}_N): \cE_0^\RR \to \cE_0^\RR
$$ 
defined by $\Gamma_N (\boldsymbol{\omega}_N)(f) =\Gamma_N^f(.|\boldsymbol{\omega}_N)$. Then 
$\Gamma_N(\boldsymbol{\omega}_N)$ is a continuous injective group morphism.
\end{theorem}

The uniqueness property of the General Multiple Gamma Hierarchy associated to $f$, proves that we obtain the same hierarchy,
shifting the index by $1$, if we use $\Gamma_1^f(.,\omega_1)$ instead of $f$, more precisely we have:

\begin{proposition}

Under the assumptions of Theorem \ref{thm:existence+uniqueness_multiple_f(1)=1}, if we denote $\boldsymbol{\omega}'_{N-1} =(\omega_2,\ldots, \omega_N)$, we have
\begin{equation*}
 \Gamma_{N-1}^{\Gamma_1^f(.,\omega_1)}(.|\boldsymbol{\omega}'_N)=\Gamma_N^f(.|\boldsymbol{\omega}_N)
\end{equation*}
\end{proposition}

\textbf{Complex parameters.}
\medskip

We study now  the non-real-analytic case for  complex parameters  $\omega_1, \ldots , \omega_N\in \CC_+$.
In general we don't have uniqueness as in Theorem \ref{thm:existence+uniqueness_multiple_f(1)=1}. 
We consider $f$ a LLD meromorphic function 
in $\CC$ of finite order with $f(1)=1$ and study the question of existence and uniqueness of a general multiple 
Gamma functions hierarchy as in Definition \ref{def:General_Multiple_Gamma_Hierarchy} with the normalization
$$
\Gamma^{f}_N(1|\boldsymbol{\omega}_N) =1 \ .
$$

Without imposing the real analyticity condition, we have the following result:

\begin{theorem}\label{thm:general2_multiple_complex}
Let $\bld \omega=(\omega_1, \omega_2, \ldots )\in \CC^\infty_+$ and 
$f$ a LLD meromorphic function of finite order such that $f(1)=1$.
There exists General Multiple Gamma Hierarchy 
$(\Gamma^{f}_N(s|\boldsymbol{\omega}_N))_{N\geq 0}$ associated to $f$, and 
for any other hierarchy $(\tilde \Gamma^{f}_N(s|\boldsymbol{\omega}_N))_{N\geq 0}$
there exists a sequence of polynomials $(P_N)_{N\geq 1}$ such that 
$$
\tilde \Gamma^{f}_N(s|\boldsymbol{\omega}_N)=\exp \left (  2 \pi i P_N(s) \right ) \Gamma^{\boldsymbol{f}}_N (s|\boldsymbol{\omega}_N)
$$
with $P_N(1) \in \ZZ$, $P_0$ is a constant integer, and for $N\geq 0$ we have
$$
\Delta_{\omega_{N+1}}  P_{N+1} =-P_N
$$
where $\Delta_\omega$ is the $\omega$-difference operator $\Delta_\omega P= P(s+\omega)-P(s)$. The polynomials $P_N$ belongs to an 
additive group of polynomials 
isomorphic to a subgroup of $\ZZ^{N+1}$.

If the functions $f$ is CLD then $\Gamma^{\boldsymbol{f}}_N(s|\boldsymbol{\omega}_N)$ and all the other solutions are CLD.
\end{theorem}

\begin{proof}
For the existence result, we carry out the same proof by induction as for Theorem \ref{thm:existence+uniqueness_multiple_f(1)=1} 
(without the normalization condition) and using  Theorem \ref{thm:general2_bis}. 
If a second solution $(\tilde \Gamma^{f}_N(s|\boldsymbol{\omega}_N))_{N\geq 1}$ 
exists, then $(\tilde \Gamma^{f}_N(s|\boldsymbol{\omega}_N)/\Gamma^{f}_N(s|\boldsymbol{\omega}_N))_{N\geq 1}$ is a solution of the 
problem for the constant function $f(s)=1$. The solution for $f(s)=1$ has no divisor and is of finite order, hence 
they are of the form $\exp(P_N)$ where $P_N$ 
are polynomial which satisfy the above difference equations.  The structure of the space of polynomials $P_N$ becomes clear from the study of the 
difference operators that follows.
\end{proof}

We observe that the integer sequence $(P_N(1))_{N\geq 1}$ of values taken at $s=1$, and the difference equation determine uniquely 
the sequence of polynomials $(P_N)_{N\geq 1}$. 

We define the $\omega$-descending factorial that form a triangular bases for the 
action of the operator $\Delta_\omega$ on polynomials.

\begin{definition}\label{def:omega-descending-factorial}
Let $\omega \in \CC^*$. For $s\in \CC$ and for an integer $k\geq 1$, we define the 
$\omega$-descending factorial as
$$
s^{[k,\omega]} = s(s-\omega)\ldots (s-(k-1)\omega)
$$
\end{definition}

For $\omega =1$ we get the usual descending factorial. A simple computation shows:

\begin{proposition}
We have
$$
\Delta_\omega s^{[k+1,\omega]}= (k+1) \omega s^{[k,\omega]}
$$
\end{proposition}

To simplify the recurrence, we write $Q_N(s)=(-1)^NP_N(s-1)$ and $a_N=Q_N(0)$. 
The polynomials $(Q_N)$ satisfy the difference equations
$$
\Delta_{\omega_{N+1}}  Q_{N+1} =Q_N \ .
$$
Now, we can give the general structure of the solutions $(Q_N)$ of this system of difference equations.

\begin{proposition}\label{prop:gen_solution}
For $N\geq 1$, the general solution of the above system of difference equations is given by 
$$
Q_N(s)=\sum_{k=0}^N\frac{a_{N-k}}{\o_1 \o_2 \ldots \o_{k}} \left [ \frac{s^{[k,\o_N]}}{k!} + A_{N,k}(\o_1, \ldots , \o_N, s) \right ]
$$
where the $A_{N,k}$ are universal polynomials with integer coefficients in $N+1$ variables and their total degree in the first $N$ 
variables is strictly less than $k$, and the 
coefficient $a_0, a_1,\ldots, a_N$ are arbitrary integers.
\end{proposition}

From this Proposition it is clear that the space of solutions $Q_N$, and the one of $P_N$, is isomorphic to a quotient subgroup of $\ZZ^{N+1}$ 
by the kernel of the map $(a_0, a_1, \ldots , a_N) \in \ZZ^{N+1}\mapsto Q_N$. The proof 
of this Proposition follows by induction on $N\geq 1$, solving the difference equation
$$
\Delta_{\omega_{N+1}}  Q_{N+1} =Q_N \ .
$$
For this, we develop the polynomials 
$$
\frac{s^{[k,\o_N]}}{k!} + A_{N,k}(\o_1, \ldots , \o_N, s)
$$
in the bases $(s^k)$, then we change to the bases $(s^{[\o_{N+1},k]})$ using the following Lemma:
\begin{lemma}
For $n\geq 1$,
$$
s^n =\sum_{k=0}^n B_{n,k}(\o) s^{[\o,k]}
$$
where $B_{n,n}=1$, $B_{n,k}\in \ZZ[X]$ and $\deg B_{n,k} \leq n-k$.
\end{lemma}
\begin{proof}
We proceed by induction. The result is clear for $n=1$, and  developing $ s^{[\o,n]} =s(s-\o)\ldots (s-(n-1)\o)$ we get
$$
s^n=s^{[\o,n]} - \sum_{k=1}^n b_k \, \o^k s^{n-k}
$$
with $b_k\in \ZZ$, and the induction hypothesis proves the result.
\end{proof}
Observe than if in this change of bases  we keep track of the total degree of the monomials on the variables $s$ and $\omega$, it is constant. Hence, 
when we chage to the bases $(s^{[\o_{N+1},k]})$ we have the degree property of the polynomials $A_{N,k}$.

Now we can study uniqueness conditions. If we assume some algebraic independence condition on the parameters, 
we have a generalization by induction of the uniqueness result from Proposition \ref{prop:unique}.

\begin{theorem}
Under the same conditions as in Theorem \ref{thm:general2_multiple_complex}, and if we assume 
that for  $1\leq n\leq N$,
\begin{equation}\label{eq:condition}
[\QQ[\omega_1,\ldots, \omega_n]:\QQ[\omega_1,\ldots, \omega_{n-1}]]\geq n+1
\end{equation}
then the hierarchy $(\Gamma^{f}_n(s|\boldsymbol{\omega}_n))_{1\leq n\leq N}$ is uniquely determined by any value 
$\Gamma^{f}_N(r|\boldsymbol{\omega}_N) $ at some rational value $r\not=1$ and $r>0$.
\end{theorem}

\begin{proof}
If we have two solutions $(g_n(s|\boldsymbol{\omega}_n))_{1\leq n\leq N}$ and 
$(\tilde g_n(s|\boldsymbol{\omega}_n))_{1\leq n\leq N}$,
the equality, $g_N(r|\boldsymbol{\omega}_N)=\tilde g_N(r|\boldsymbol{\omega}_N)$, at the rational value $r \in \CC_+$, 
that is neither a zero nor pole of the functions that are LLD,
shows that the corresponding polynomials $Q_N$ and $\tilde Q_N$ satisfy
$$
Q_N(r+1)-\tilde Q_N(r+1) =a\in \ZZ
$$
Then, using Proposition \ref{prop:gen_solution}, this gives 
$$
\sum_{k=0}^N\frac{a_{N-k}-\tilde a_{N-k}}{\o_1 \o_2 \ldots \o_{k}} \left [ \frac{(r+1)^{[k,\o_N]}}{k!} + A_{N,k}(\o_1, \ldots , \o_N, r+1) \right ] =a
$$
or, multiplying by $\o_1 \o_2 \ldots \o_{N}$, we get the algebraic relation
$$
(a_0-\tilde a_0) \frac{(r+1)^{[N,\o_N]}}{N!} + \ldots + (a_N-\tilde a_N-a)\o_1 \o_2 \ldots \o_{N}=0
$$
where the dots part is a polynomial of degre $<N$ in $\o_N$ with coefficients in $\QQ[\omega_1,\ldots, \omega_{N-1}]$. 
Since $[\QQ[\omega_1,\ldots, \omega_N]:\QQ[\omega_1,\ldots, \omega_{N-1}]]\geq N+1$ we must have $a_0=\tilde a_0$. 
Therefore we have $g_1(s|\o_1)=\tilde g_1(s|\o_1)$.
Using the induction hypothesis on $N$ (replacing $f$ 
by $g_1(s|\o_1)=\tilde g_1(s|\o_1)$), we then get by induction  
$a_1 =\tilde a_1$,..., $a_N =\tilde a_N$.
\end{proof}

To conclude this section, we note that Ruijsenaars (2000, \cite{Rui2000}) exploited also the difference 
equations and their minimal solutions to prove numerous properties of  Barnes Multiple Gamma functions.
Shintani (1976, \cite{Shi1}) extended Barnes approach to Multiple Gamma functions to a several variable setting. 
Friedman and Ruijsenaars (2004, \cite{Fri-Rui2004}) extended Shintani's Multiple Gamma functions. We can also apply our 
functional equation approach
to define these several variables Gamma functions without Barnes-Hurwitz zeta functions and we will study 
these functions in a forthcoming article.

\bigskip

\textbf{Acknowledgements.} I am grateful to Prof. David Bl\'azquez for his comments and corrections.

\end{document}